\documentclass[11pt]{amsart}
\usepackage{amsmath,amssymb}
\newtheorem{theorem}{Theorem}[section]
\newtheorem{proposition}[theorem]{Proposition}
\newtheorem{corollary}[theorem]{Corollary}
\newtheorem{lemma}[theorem]{Lemma}

\theoremstyle{definition}

\begin{document}

\title[Geometric inequalities and the ABP technique]{Geometric inequalities and the Alexandrov-Bakelman-Pucci technique}
\author{Simon Brendle}
\address{Department of Mathematics \\ Columbia University \\ New York NY 10027}
\begin{abstract}
In this expository paper, we discuss a unified framework for proving various geometric inequalities, based on the so-called Alexandrov-Bakelman-Pucci technique. Examples include Cabr\'e's proof of the classical isoperimetric inequality in Euclidean space; the Fenchel-Willmore-Chen inequality for the mean curvature of a submanifold; the sharp version of the Michael-Simon Sobolev inequality for submanifolds; the sharp version of Ecker's logarithmic Sobolev inequality for submanifolds; and the Sobolev inequality for complete manifolds with nonnegative Ricci curvature and Euclidean volume growth. Finally, we discuss a connection to the work of Heintze and Karcher on the volume of a tubular neighborhood of a hypersurface in a manifold with nonnegative Ricci curvature.
\end{abstract}
\thanks{The author was supported by the National Science Foundation under grant DMS-2403981 and by the Simons Foundation.}

\maketitle 

\section{The classical isoperimetric inequality and the Sobolev inequality}

\label{Sobolev}

The isoperimetric inequality is one of the most important inequalities in differential geometry. 

\begin{theorem}[Isoperimetric inequality]
\label{isoperimetric.inequality}
Let $D$ be a compact domain in $\mathbb{R}^n$ with smooth boundary. Then $|\partial D| \geq n \, |B^n|^{\frac{1}{n}} \, |D|^{\frac{n-1}{n}}$, where $B^n$ denotes the open unit ball in $\mathbb{R}^n$.
\end{theorem}

The isoperimetric inequality is closely related to the Brunn-Minkowski inequality (see \cite{Lusternik}, \cite{Stein-Shakarchi2}) and to the Sobolev inequality.

\begin{theorem}[Sobolev inequality]
\label{Sobolev.inequality}
Let $D$ be a compact domain in $\mathbb{R}^n$ with smooth boundary. Let $f$ be a smooth positive function on $D$. Then 
\[\int_D |\nabla f| + \int_{\partial D} f \geq n \, |B^n|^{\frac{1}{n}} \, \Big ( \int_D f^{\frac{n}{n-1}} \Big )^{\frac{n-1}{n}},\] 
where $B^n$ denotes the open unit ball in $\mathbb{R}^n$.
\end{theorem}

Many proofs of the isoperimetric inequality have been found. These employ a variety of techniques, including symmetrization techniques, variational methods, measure transportation (see \cite{Gromov}, \cite{McCann1}, \cite{McCann2}, \cite{McCann-Guillen}, \cite{Trudinger}), and the Alexandrov-Bakelman-Pucci technique (see \cite{Cabre1}, \cite{Cabre2}, \cite{Cabre3}). 

In the measure transportation approach, it is convenient to normalize the domain $D$ so that $|D| = |B^n|$. The key idea is to construct a measure-preserving map $\Phi$ from the interior of $D$ to the open unit ball $B^n$ with the property that the eigenvalues of the differential of $\Phi$ are nonnegative real numbers. A map $\Phi$ with these properties can be constructed in several ways. Gromov's work \cite{Gromov} uses the Knothe map, which has the additional property that the differential of $\Phi$ takes values in the space of upper triangular matrices. An alternative approach based on optimal transport was developed independently by McCann \cite{McCann1},\cite{McCann2} and Trudinger \cite{Trudinger}. In this approach, one defines $\Phi(x) = \nabla u(x)$, where $u$ is a convex function solving the second boundary value problem for the Monge-Amp\`ere equation. For this choice of $\Phi$, the differential of $\Phi$ takes values in the set of symmetric matrices. 

In \cite{Cabre2},\cite{Cabre3}, Cabr\'e gave an elegant proof of the isoperimetric inequality based on what has become known as the Alexandrov-Bakelman-Pucci technique. The Alexandrov-Bakelman-Pucci technique in Cabr\'e's work is, in a certain sense, dual to the optimal transport approach developed by McCann and Trudinger. In the optimal transport approach, one considers a convex function satisfying $\det D^2 u = 1$ and concludes that $\Delta u \geq n$ by the arithmetic-geometric mean inequality. In the Alexandrov-Bakelman-Pucci approach, one considers a solution of $\Delta u = n$ and concludes that $\det D^2 u \leq 1$ at each point where the Hessian of $u$ is weakly positive definite. The Alexandrov-Bakelman-Pucci technique only requires existence and regularity results for linear partial differential equations. 

In the remainder of this section, we describe a proof of Theorem \ref{Sobolev.inequality} following the arguments in Cabr\'e's work \cite{Cabre2},\cite{Cabre3}.

We first consider the special case that $D$ is connected. By multiplying $f$ by a suitable constant, we may arrange that 
\begin{equation} 
\label{Sobolev.normalization}
\int_D |\nabla f| + \int_{\partial D} f = n \int_D f^{\frac{n}{n-1}}. 
\end{equation}
Since $D$ is connected, we can find a function $u: D \to \mathbb{R}$ such that 
\begin{equation} 
\label{Sobolev.pde.for.u}
\text{\rm div}(f \, \nabla u) = n \, f^{\frac{n}{n-1}} - |\nabla f| 
\end{equation}
at each point in $D$ and 
\begin{equation} 
\label{Sobolev.boundary.condition.for.u}
\langle \nabla u,\eta \rangle = 1 
\end{equation}
at each point on $\partial D$, where $\eta$ denotes the outward-pointing unit normal vector field to $\partial D$. The function $u$ is of class $C^{2,\gamma}$ for each $0 < \gamma < 1$. We define 
\begin{align*} 
U &:= \{x \in D \setminus \partial D: |\nabla u(x)| < 1\}, \\ 
A &:= \{x \in U: D^2 u(x) \geq 0\}. 
\end{align*} 
Moreover, we define a map $\Phi: U \to \mathbb{R}^n$ by $\Phi(x) = \nabla u(x)$ for all $x \in U$. 

\begin{lemma}
\label{Sobolev.surjectivity} 
The open unit ball $B^n$ is contained in the image $\Phi(A)$.
\end{lemma}

\textbf{Proof.} 
Let us fix an arbitrary vector $\xi \in \mathbb{R}^n$ such that $|\xi|<1$. We define a function $w: D \to \mathbb{R}$ by $w(x) := u(x) - \langle x,\xi \rangle$. Using (\ref{Sobolev.boundary.condition.for.u}), we obtain 
\[\langle \nabla w(x),\eta(x) \rangle = \langle \nabla u(x),\eta(x) \rangle - \langle \eta(x),\xi \rangle = 1 - \langle \eta(x),\xi \rangle > 0\] 
for each point $x \in \partial D$. Consequently, we can find a point $\bar{x} \in D \setminus \partial D$ where the function $w$ attains its minimum. Clearly, $\nabla w(\bar{x}) = 0$ and $D^2 w(\bar{x}) \geq 0$. From this, we deduce that $\xi = \nabla u(\bar{x})$ and $D^2 u(\bar{x}) \geq 0$. This implies $\bar{x} \in A$ and $\Phi(\bar{x}) = \xi$. Thus, $B^n \subset \Phi(A)$. This completes the proof of Lemma \ref{Sobolev.surjectivity}. \\

\begin{lemma}
\label{Sobolev.Jacobian.estimate}
The Jacobian determinant of $\Phi$ satisfies 
\[0 \leq \det D\Phi(x) \leq f(x)^{\frac{n}{n-1}}\] 
for each point $x \in A$.
\end{lemma}

\textbf{Proof.} 
Let us fix a point $x \in A$. Using (\ref{Sobolev.pde.for.u}), we obtain 
\begin{align*} 
\Delta u(x) 
&= n \, f(x)^{\frac{1}{n-1}} - f(x)^{-1} \, |\nabla f(x)| - f(x)^{-1} \, \langle \nabla f(x),\nabla u(x) \rangle \\ 
&\leq n \, f(x)^{\frac{1}{n-1}}. 
\end{align*}
Moreover, $D^2 u(x) \geq 0$ since $x \in A$. Hence, the arithmetic-geometric mean inequality implies 
\[0 \leq \det D^2 u(x) \leq \Big ( \frac{\Delta u(x)}{n} \Big )^n \leq f(x)^{\frac{n}{n-1}}.\] 
This completes the proof of Lemma \ref{Sobolev.Jacobian.estimate}. \\

Using Lemma \ref{Sobolev.surjectivity} and Lemma \ref{Sobolev.Jacobian.estimate}, we obtain 
\[|B^n| \leq \int_A |\det D\Phi(x)| \, dx \leq \int_D f(x)^{\frac{n}{n-1}} \, dx.\] 
Combining this inequality with the normalization condition (\ref{Sobolev.normalization}), we conclude that 
\[\int_D |\nabla f| + \int_{\partial D} f = n \int_D f^{\frac{n}{n-1}} \geq n \, |B^n|^{\frac{1}{n}} \, \Big ( \int_D f^{\frac{n}{n-1}} \Big )^{\frac{n-1}{n}}.\] 
This completes the proof of Theorem \ref{Sobolev.inequality} in the special case when $D$ is connected.

Finally, if $D$ is disconnected, we apply the inequality to each connected component of $D$, and take the sum over all connected components. \\

\section{The Fenchel-Willmore-Chen inequality for submanifolds of Euclidean space} 

\label{Fenchel/Willmore.Chen.estimate}

In this section, we discuss how the well-known Fenchel-Willmore-Chen inequality fits within the framework of the Alexandrov-Bakelman-Pucci technique.

\begin{theorem}[Fenchel-Willmore-Chen inequality] 
\label{FWC.estimate}
Let $\Sigma$ be a compact $n$-dimensional submanifold of $\mathbb{R}^{n+m}$ without boundary. Then 
\[\int_\Sigma \Big ( \frac{|H|}{n} \Big )^n \geq |S^n|,\] 
where $H$ denotes the mean curvature vector of $\Sigma$.
\end{theorem}

Theorem \ref{FWC.estimate} was proved by Chen \cite{Chen}, building upon earlier work of Fenchel \cite{Fenchel} and Willmore \cite{Willmore}.

In the remainder of this section, we describe the proof of Theorem \ref{FWC.estimate}. We define 
\begin{align*} 
U &:= \{(x,y): x \in \Sigma, \, y \in T_x^\perp \Sigma, \, |y| < 1\}, \\ 
A &:= \{(x,y) \in U: -\langle I\!I(x),y \rangle \geq 0\}. 
\end{align*} 
Moreover, we define a map $\Phi: U \to \mathbb{R}^{n+m}$ by $\Phi(x,y) = y$ for all $(x,y) \in U$. 

\begin{lemma}
\label{FWC.surjectivity} 
The open unit ball $B^{n+m}$ is contained in the image $\Phi(A)$.
\end{lemma}

\textbf{Proof.} 
Let us fix an arbitrary vector $\xi \in \mathbb{R}^{n+m}$ such that $|\xi|<1$. We define a function $w: \Sigma \to \mathbb{R}$ by $w(x) := -\langle x,\xi \rangle$. We can find a point $\bar{x} \in \Sigma$ where the function $w$ attains its minimum. Clearly, $\nabla^\Sigma w(\bar{x}) = 0$ and $D_\Sigma^2 w(\bar{x}) \geq 0$. From this, we deduce that $\xi$ is orthogonal to the tangent space $T_{\bar{x}} \Sigma$ and $-\langle I\!I(\bar{x}),\xi \rangle \geq 0$. Therefore, if we define a normal vector $\bar{y} \in T_{\bar{x}}^\perp \Sigma$ by $\bar{y} = \xi$, then $(\bar{x},\bar{y}) \in A$ and $\Phi(\bar{x},\bar{y}) = \xi$. Thus, $B^{n+m} \subset \Phi(A)$. This completes the proof of Lemma \ref{FWC.surjectivity}. \\

\begin{lemma}
\label{FWC.Jacobian.determinant}
The Jacobian determinant of $\Phi$ is given by 
\[\det D\Phi(x,y) = \det (-\langle I\!I(x),y \rangle)\] 
for each point $(x,y) \in U$.
\end{lemma}

\textbf{Proof.} 
Let us fix a point $(\bar{x},\bar{y}) \in U$. Let $\{e_1,\hdots,e_n\}$ be an orthonormal basis of the tangent space $T_{\bar{x}} \Sigma$, and let $(x_1,\hdots,x_n)$ be a local coordinate system on $\Sigma$ such that $\frac{\partial}{\partial x_i} = e_i$ at the point $\bar{x}$. Moreover, let $\{\nu_1,\hdots,\nu_m\}$ denote a local orthonormal frame for the normal bundle $T^\perp \Sigma$ such that $\frac{\partial}{\partial x_i} \nu_\alpha \in T_{\bar{x}} \Sigma$ at the point $\bar{x}$. Every normal vector $y$ can be written in the form $y = \sum_{\alpha=1}^m y_\alpha \nu_\alpha$. With this understood, $(x_1,\hdots,x_n,y_1,\hdots,y_m)$ is a local coordinate system on the total space of the normal bundle $T^\perp \Sigma$. A straightforward calculation gives 
\[\frac{\partial \Phi}{\partial x_i} = -\sum_{j=1}^n \langle I\!I(e_i,e_j),\bar{y} \rangle \, e_j\] 
and 
\[\frac{\partial \Phi}{\partial y_\alpha} = \nu_\alpha\] 
at the point $(\bar{x},\bar{y})$. Therefore, the differential of $\Phi$ at the point $(\bar{x},\bar{y})$ is given by 
\[D\Phi(\bar{x},\bar{y}) = \begin{bmatrix} -\langle I\!I(\bar{x}),\bar{y} \rangle & 0 \\ 0 & \text{\rm id} \end{bmatrix}.\] 
The assertion follows by taking the determinant. This completes the proof of Lemma \ref{FWC.Jacobian.determinant}. \\

\begin{lemma}
\label{FWC.Jacobian.estimate}
The Jacobian determinant of $\Phi$ satisfies 
\[0 \leq \det D\Phi(x,y) \leq \Big ( -\frac{\langle H(x),y \rangle}{n} \Big )^n\] 
for each point $(x,y) \in A$.
\end{lemma}

\textbf{Proof.} 
Let us fix a point $(x,y) \in A$. Then $-\langle I\!I(x),y \rangle \geq 0$. Hence, the arithmetic-geometric mean inequality implies
\[0 \leq \det (-\langle I\!I(x),y \rangle) \leq \Big ( -\frac{\langle H(x),y \rangle}{n} \Big )^n.\] 
The assertion follows now from Lemma \ref{FWC.Jacobian.determinant}. This completes the proof of Lemma \ref{FWC.Jacobian.estimate}. \\

\begin{lemma} 
\label{integral}
We have 
\[\int_{\{y \in \mathbb{R}^m: |y| < 1\}} (-\langle a,y \rangle)_+^n \, dy = \frac{|B^{n+m}|}{|S^n|} \, |a|^n\] 
for every vector $a \in \mathbb{R}^m$. 
\end{lemma} 

\textbf{Proof.} 
It suffices to prove the assertion in the special case when $a \in \mathbb{R}^m$ is a unit vector. If $a \in \mathbb{R}^m$ is a unit vector, we obtain 
\begin{align*} 
\int_{\{y \in \mathbb{R}^m: |y| < 1\}} (-\langle a,y \rangle)_+^n \, dy 
&= |B^{m-1}| \int_0^1 s^n \, (1-s^2)^{\frac{m-1}{2}} \, ds \\ 
&= \frac{1}{2} \, |B^{m-1}| \int_0^1 t^{\frac{n-1}{2}} \, (1-t)^{\frac{m-1}{2}} \, dt \\ 
&= \frac{\Gamma(\frac{n+1}{2}) \, \Gamma(\frac{m+1}{2})}{2 \, \Gamma(\frac{n+m+2}{2})} \, |B^{m-1}|. 
\end{align*}
The first equality follows from Fubini's theorem. In the second equality, we have used the substitution $t=s^2$. The third equality follows from the classical formula relating the Beta function to the Gamma function (see e.g. \cite{Stein-Shakarchi1}, pp.~175--176). Using the identities 
\[|S^n| = \frac{2 \, \pi^{\frac{n+1}{2}}}{\Gamma(\frac{n+1}{2})},\] 
\[|B^{m-1}| = \frac{\pi^{\frac{m-1}{2}}}{\Gamma(\frac{m+1}{2})},\] 
and 
\[|B^{n+m}| = \frac{\pi^{\frac{n+m}{2}}}{\Gamma(\frac{n+m+2}{2})},\] 
we obtain 
\[\frac{\Gamma(\frac{n+1}{2}) \, \Gamma(\frac{m+1}{2})}{2 \, \Gamma(\frac{n+m+2}{2})} \, |B^{m-1}| = \frac{|B^{n+m}|}{|S^n|}.\] 
This completes the proof of Lemma \ref{integral}. \\

Using Lemma \ref{FWC.surjectivity} and Lemma \ref{FWC.Jacobian.estimate}, we obtain 
\begin{align*} 
|B^{n+m}| 
&\leq \int_\Sigma \bigg ( \int_{T_x^\perp \Sigma} |\det D\Phi(x,y)| \, 1_A(x,y) \, dy \bigg ) \, d\text{\rm vol}(x) \\ 
&\leq \int_\Sigma \bigg ( \int_{\{y \in T_x^\perp \Sigma: |y| < 1\}} \Big ( -\frac{\langle H(x),y \rangle}{n} \Big )_+^n \, dy \bigg ) \, d\text{\rm vol}(x). 
\end{align*} 
On the other hand, Lemma \ref{integral} implies that 
\[\int_{\{y \in T_x^\perp \Sigma: |y| < 1\}} (-\langle H(x),y \rangle)_+^n \, dy = \frac{|B^{n+m}|}{|S^n|} \, |H(x)|^n\] 
for each point $x \in \Sigma$. Putting these facts together, we conclude that 
\[|B^{n+m}| \leq \frac{|B^{n+m}|}{|S^n|} \int_\Sigma \Big ( \frac{|H(x)|}{n} \Big )^n \, d\text{\rm vol}(x).\] 
This completes the proof of Theorem \ref{FWC.estimate}. \\

\section{The Sobolev inequality for submanifolds of Euclidean space} 

\label{Michael.Simon}

In this section, we explain how the Alexandrov-Bakelman-Pucci technique can be used to prove an isoperimetric inequality for submanifolds. 

\begin{theorem}[cf. \cite{Brendle1}]
\label{Michael.Simon.inequality}
Let $\Sigma$ be a compact $n$-dimensional submanifold of $\mathbb{R}^{n+m}$ (possibly with boundary), where $m \geq 2$. Let $f$ be a smooth positive function on $\Sigma$. Then 
\[\int_\Sigma \sqrt{|\nabla^\Sigma f|^2 + f^2 \, |H|^2} + \int_{\partial \Sigma} f \geq n \, \Big ( \frac{(n+m) \, |B^{n+m}|}{m \, |B^m|} \Big )^{\frac{1}{n}} \, \Big ( \int_\Sigma f^{\frac{n}{n-1}} \Big )^{\frac{n-1}{n}},\] 
where $H$ denotes the mean curvature vector of $\Sigma$.
\end{theorem}

The case $m=2$ is of particular interest. Using Theorem \ref{Michael.Simon.inequality} and the formula $(n+2) \, |B^{n+2}| = 2\pi \, |B^n| = 2 \, |B^2| \, |B^n|$, we obtain a sharp Sobolev inequality for submanifolds of codimension $2$.

\begin{corollary}[cf. \cite{Brendle1}]
\label{Michael.Simon.inequality.codimension.2}
Let $\Sigma$ be a compact $n$-dimensional submanifold of $\mathbb{R}^{n+2}$ (possibly with boundary). Let $f$ be a smooth positive function on $\Sigma$. Then 
\[\int_\Sigma \sqrt{|\nabla^\Sigma f|^2 + f^2 \, |H|^2} + \int_{\partial \Sigma} f \geq n \, |B^n|^{\frac{1}{n}} \, \Big ( \int_\Sigma f^{\frac{n}{n-1}} \Big )^{\frac{n-1}{n}},\] 
where $H$ denotes the mean curvature vector of $\Sigma$.
\end{corollary}

In the special case when $\Sigma$ is a minimal submanifold, we can draw the following conclusion.

\begin{corollary}[cf. \cite{Brendle1}]
\label{isoperimetric.inequality.codimension.2}
Let $\Sigma$ be a compact $n$-dimensional minimal submanifold of $\mathbb{R}^{n+2}$ (possibly with boundary). Then $|\partial \Sigma| \geq n \, |B^n|^{\frac{1}{n}} \, |\Sigma|^{\frac{n-1}{n}}$.
\end{corollary}

Corollary \ref{Michael.Simon.inequality.codimension.2} and Corollary \ref{isoperimetric.inequality.codimension.2} immediately imply the corresponding inequalities for submanifolds of codimension $1$. We refer to \cite{Brendle4} for a direct proof in the codimension $1$ setting.

The isoperimetric inequality for minimal surfaces has a long history, going back to the fundamental work of Carleman \cite{Carleman} in 1921. A number of partial results have been established over the years; see \cite{Almgren}, \cite{Choe}, \cite{Feinberg}, \cite{Hsiung}, \cite{Li-Schoen-Yau}, \cite{Osserman-Schiffer}, \cite{Reid}, \cite{Stone}. In particular, Almgren \cite{Almgren} showed that the sharp isoperimetric inequality holds for area-minimizing submanifolds of arbitrary dimension and codimension. Almgren's proof relies on arguments from geometric measure theory together with a generalization of the Fenchel-Willmore-Chen inequality.

Corollary \ref{Michael.Simon.inequality.codimension.2} can be viewed as a sharp version of the famous Michael-Simon Sobolev inequality. The non-sharp version of the inequality was proved independently by Allard \cite{Allard} and by Michael and Simon \cite{Michael-Simon} (see also \cite{Castillon}).

In the remainder of this section, we describe a proof of Theorem \ref{Michael.Simon.inequality} using the Alexandrov-Bakelman-Pucci technique, following the arguments in \cite{Brendle1}. We refer to \cite{Brendle-Eichmair} for an alternative approach based on an optimal transport problem between spaces of unequal dimension.

We first consider the special case that $\Sigma$ is connected. By multiplying $f$ by a suitable constant, we may arrange that 
\begin{equation} 
\label{Michael.Simon.normalization}
\int_\Sigma \sqrt{|\nabla^\Sigma f|^2 + f^2 \, |H|^2} + \int_{\partial \Sigma} f = n \int_\Sigma f^{\frac{n}{n-1}}. 
\end{equation}
Since $\Sigma$ is connected, we can find a function $u: \Sigma \to \mathbb{R}$ with the property that 
\begin{equation} 
\label{Michael.Simon.pde.for.u}
\text{\rm div}_\Sigma(f \, \nabla^\Sigma u) = n \, f^{\frac{n}{n-1}} - \sqrt{|\nabla^\Sigma f|^2 + f^2 \, |H|^2} 
\end{equation}
at each point on $\Sigma$ and 
\begin{equation} 
\label{Michael.Simon.boundary.condition.for.u}
\langle \nabla^\Sigma u,\eta \rangle = 1 
\end{equation}
at each point on $\partial \Sigma$, where $\eta$ denotes the outward-pointing unit normal vector field to $\partial \Sigma$ in $\Sigma$. The function $u$ is of class $C^{2,\gamma}$ for each $0 < \gamma < 1$. We define 
\begin{align*} 
\Omega &:= \{x \in \Sigma \setminus \partial \Sigma: |\nabla^\Sigma u(x)| < 1\}, \\ 
U &:= \{(x,y): x \in \Sigma \setminus \partial \Sigma, \, y \in T_x^\perp \Sigma, \, |\nabla^\Sigma u(x)|^2 + |y|^2 < 1\}, \\ 
A &:= \{(x,y) \in U: D_\Sigma^2 u(x) - \langle I\!I(x),y \rangle \geq 0\}. 
\end{align*} 
Moreover, we define a map $\Phi: U \to \mathbb{R}^{n+m}$ by $\Phi(x,y) = \nabla^\Sigma u(x) + y$ for all $(x,y) \in U$. Note that $|\Phi(x,y)|^2 = |\nabla^\Sigma u(x)|^2+|y|^2 < 1$ for all $(x,y) \in U$.

\begin{lemma}
\label{Michael.Simon.surjectivity}
The open unit ball $B^{n+m}$ is contained in the image $\Phi(A)$.
\end{lemma}

\textbf{Proof.} 
Let us fix an arbitrary vector $\xi \in \mathbb{R}^{n+m}$ such that $|\xi|<1$. We define a function $w: \Sigma \to \mathbb{R}$ by $w(x) := u(x) - \langle x,\xi \rangle$. Using (\ref{Michael.Simon.boundary.condition.for.u}), we obtain 
\[\langle \nabla^\Sigma w(x),\eta(x) \rangle = \langle \nabla^\Sigma u(x),\eta(x) \rangle - \langle \eta(x),\xi \rangle = 1 - \langle \eta(x),\xi \rangle > 0\] 
for each point $x \in \partial \Sigma$. Consequently, we can find a point $\bar{x} \in \Sigma \setminus \partial \Sigma$ where the function $w$ attains its minimum. Clearly, $\nabla^\Sigma w(\bar{x}) = 0$ and $D_\Sigma^2 w(\bar{x}) \geq 0$. Therefore, we can find a vector $\bar{y} \in T_{\bar{x}}^\perp \Sigma$ such that $\xi = \nabla^\Sigma u(\bar{x}) + \bar{y}$ and $D_\Sigma^2 u(\bar{x}) - \langle I\!I(\bar{x}),\bar{y} \rangle \geq 0$. This implies $(\bar{x},\bar{y}) \in A$ and $\Phi(\bar{x},\bar{y}) = \xi$. Thus, $B^{n+m} \subset \Phi(A)$. This completes the proof of Lemma \ref{Michael.Simon.surjectivity}. \\

\begin{lemma} 
\label{Michael.Simon.Jacobian.determinant}
The Jacobian determinant of $\Phi$ is given by 
\[\det D\Phi(x,y) = \det (D_\Sigma^2 u(x) - \langle I\!I(x),y \rangle)\] 
for all $(x,y) \in U$.
\end{lemma}

\textbf{Proof.} 
Let us fix a point $(\bar{x},\bar{y}) \in U$. Let $\{e_1,\hdots,e_n\}$ be an orthonormal basis of the tangent space $T_{\bar{x}} \Sigma$, and let $(x_1,\hdots,x_n)$ be a local coordinate system on $\Sigma$ such that $\frac{\partial}{\partial x_i} = e_i$ at the point $\bar{x}$. Moreover, let $\{\nu_1,\hdots,\nu_m\}$ denote a local orthonormal frame for the normal bundle $T^\perp \Sigma$ such that $\frac{\partial}{\partial x_i} \nu_\alpha \in T_{\bar{x}} \Sigma$ at the point $\bar{x}$. Every normal vector $y$ can be written in the form $y = \sum_{\alpha=1}^m y_\alpha \nu_\alpha$. With this understood, $(x_1,\hdots,x_n,y_1,\hdots,y_m)$ is a local coordinate system on the total space of the normal bundle $T^\perp \Sigma$. A straightforward calculation gives 
\[\frac{\partial \Phi}{\partial x_i} = \sum_{j=1}^n ((D_\Sigma^2 u)(e_i,e_j) - \langle I\!I(e_i,e_j),\bar{y} \rangle) \, e_j + I\!I(e_i,\nabla^\Sigma u)\] 
and 
\[\frac{\partial \Phi}{\partial y_\alpha} = \nu_\alpha\] 
at the point $(\bar{x},\bar{y})$. Therefore, the differential of $\Phi$ at the point $(\bar{x},\bar{y})$ takes the form 
\[D\Phi(\bar{x},\bar{y}) = \begin{bmatrix} D_\Sigma^2 u(\bar{x}) - \langle I\!I(\bar{x}),\bar{y} \rangle & 0 \\ * & \text{\rm id} \end{bmatrix}.\] 
The assertion follows by taking the determinant. This completes the proof of Lemma \ref{Michael.Simon.Jacobian.determinant}. \\

\begin{lemma} 
\label{Michael.Simon.Jacobian.estimate}
The Jacobian determinant of $\Phi$ satisfies 
\[0 \leq \det D\Phi(x,y) \leq f(x)^{\frac{n}{n-1}}\] 
for all $(x,y) \in A$.
\end{lemma}

\textbf{Proof.} 
Let us fix a point $(x,y) \in A$. Then $|\nabla^\Sigma u(x)|^2+|y|^2 < 1$. Using the Cauchy-Schwarz inequality, we obtain 
\begin{align*} 
&-\langle \nabla^\Sigma f(x),\nabla^\Sigma u(x) \rangle - f(x) \, \langle H(x),y \rangle \\ 
&\leq \sqrt{|\nabla^\Sigma f(x)|^2 + f(x)^2 \, |H(x)|^2} \, \sqrt{|\nabla^\Sigma u(x)|^2+|y|^2} \\ 
&\leq \sqrt{|\nabla^\Sigma f(x)|^2 + f(x)^2 \, |H(x)|^2}. 
\end{align*}
Using (\ref{Michael.Simon.pde.for.u}), we obtain 
\begin{align*} 
&\Delta_\Sigma u(x) - \langle H(x),y \rangle \\ 
&= n \, f(x)^{\frac{1}{n-1}} - f(x)^{-1} \, \sqrt{|\nabla^\Sigma f(x)|^2 + f(x)^2 \, |H(x)|^2} \\ 
&- f(x)^{-1} \, \langle \nabla^\Sigma f(x),\nabla^\Sigma u(x) \rangle - \langle H(x),y \rangle \\ 
&\leq n \, f(x)^{\frac{1}{n-1}}. 
\end{align*} 
Moreover, $D_\Sigma^2 u(x) - \langle I\!I(x),y \rangle \geq 0$ since $(x,y) \in A$. Hence, the arithmetic-geometric mean inequality implies
\[0 \leq \det (D_\Sigma^2 u(x) - \langle I\!I(x),y \rangle) \leq \Big ( \frac{\text{\rm tr}(D_\Sigma^2 u(x) - \langle I\!I(x),y \rangle)}{n} \Big )^n \leq f(x)^{\frac{n}{n-1}}.\] 
The assertion follows now from Lemma \ref{Michael.Simon.Jacobian.determinant}. This completes the proof of Lemma \ref{Michael.Simon.Jacobian.estimate}. \\ 

Using Lemma \ref{Michael.Simon.surjectivity} and Lemma \ref{Michael.Simon.Jacobian.estimate}, we obtain 
\begin{align*} 
&|B^{n+m}| \, (1-\sigma^{n+m}) \\ 
&= \int_{\{\xi \in \mathbb{R}^{n+m}: \sigma^2 < |\xi|^2 < 1\}} 1 \, d\xi \\ 
&\leq \int_\Omega \bigg ( \int_{\{y \in T_x^\perp \Sigma: \sigma^2 < |\Phi(x,y)|^2 < 1\}} |\det D\Phi(x,y)| \, 1_A(x,y) \, dy \bigg ) \, d\text{\rm vol}(x) \\ 
&\leq \int_\Omega \bigg ( \int_{\{y \in T_x^\perp \Sigma: \sigma^2 < |\nabla^\Sigma u(x)|^2+|y|^2 < 1\}} f(x)^{\frac{n}{n-1}} \, dy \bigg ) \, d\text{\rm vol}(x) \\ 
&= |B^m| \int_\Omega \Big [ (1-|\nabla^\Sigma u(x)|^2)^{\frac{m}{2}} -  (\sigma^2-|\nabla^\Sigma u(x)|^2)_+^{\frac{m}{2}} \Big ] \, f(x)^{\frac{n}{n-1}} \, d\text{\rm vol}(x)
\end{align*} 
for all $0 \leq \sigma < 1$. Since $m \geq 2$, an elementary inequality gives $b^{\frac{m}{2}} - a^{\frac{m}{2}} \leq \frac{m}{2} \, (b-a)$ for $0 \leq a \leq b \leq 1$. Consequently, 
\begin{align*} 
&(1-|\nabla^\Sigma u(x)|^2)^{\frac{m}{2}} -  (\sigma^2-|\nabla^\Sigma u(x)|^2)_+^{\frac{m}{2}} \\ 
&\leq \frac{m}{2} \, \Big [ (1-|\nabla^\Sigma u(x)|^2) -  (\sigma^2-|\nabla^\Sigma u(x)|^2)_+ \Big ] \\ 
&\leq \frac{m}{2} \, (1-\sigma^2) 
\end{align*}
for all $x \in \Omega$ and all $0 \leq \sigma < 1$. Putting these facts together, we obtain 
\[|B^{n+m}| \, (1-\sigma^{n+m}) \leq \frac{m}{2} \, |B^m| \, (1-\sigma^2) \int_\Sigma f^{\frac{n}{n-1}}\] 
for all $0 \leq \sigma < 1$. Finally, we divide by $1-\sigma$ and take the limit as $\sigma \to 1$. Thus, 
\[(n+m) \, |B^{n+m}| \leq m \, |B^m| \int_\Sigma f^{\frac{n}{n-1}}.\] 
Combining this inequality with the normalization condition (\ref{Michael.Simon.normalization}), we conclude that 
\begin{align*} 
&\int_\Sigma \sqrt{|\nabla^\Sigma f|^2 + f^2 \, |H|^2} + \int_{\partial \Sigma} f \\ 
&= n \int_\Sigma f^{\frac{n}{n-1}} \geq n \, \Big ( \frac{(n+m) \, |B^{n+m}|}{m \, |B^m|} \Big )^{\frac{1}{n}} \, \Big ( \int_\Sigma f^{\frac{n}{n-1}} \Big )^{\frac{n-1}{n}}. 
\end{align*} 
This completes the proof of Theorem \ref{Michael.Simon.inequality} in the special case when $\Sigma$ is connected. 

Finally, if $\Sigma$ is disconnected, we apply the inequality to each connected component of $\Sigma$, and take the sum over all connected components. \\

\section{The logarithmic Sobolev inequality for submanifolds of Euclidean space} 

\label{log.Sobolev}

In this section, we discuss a sharp logarithmic Sobolev inequality for submanifolds.

\begin{theorem}[cf. \cite{Brendle2}] 
\label{log.Sobolev.inequality}
Let $\Sigma$ be a compact $n$-dimensional submanifold of $\mathbb{R}^{n+m}$ without boundary. Let $f$ be a positive smooth function on $\Sigma$. Then
\begin{align*} 
&\int_\Sigma f \, \Big ( \log f + n + \frac{n}{2} \, \log(4\pi) \Big ) - \int_\Sigma f^{-1} \, |\nabla^\Sigma f|^2 - \int_\Sigma f \, |H|^2 \\ 
&\leq \Big ( \int_\Sigma f \Big ) \, \log \Big ( \int_\Sigma f \Big ), 
\end{align*} 
where $H$ denotes the mean curvature vector of $\Sigma$.
\end{theorem}

If we put 
\[f = (4\pi)^{-\frac{n}{2}} \, e^{-\frac{1}{4} \, |x|^2} \, \varphi,\] 
then Theorem \ref{log.Sobolev.inequality} can be rewritten in the following equivalent form.

\begin{corollary}[cf. \cite{Brendle2}] 
\label{log.Sobolev.inequality.alternative.version}
Let $\Sigma$ be a compact $n$-dimensional submanifold of $\mathbb{R}^{n+m}$ without boundary, and let 
\[d\gamma = (4\pi)^{-\frac{n}{2}} \, e^{-\frac{1}{4} \, |x|^2} \, d\text{\rm vol}\] 
denote the Gaussian measure on $\Sigma$. Let $\varphi$ be a positive smooth function on $\Sigma$. Then
\begin{align*} 
&\int_\Sigma \varphi \, \log \varphi \, d\gamma - \int_\Sigma \varphi^{-1} \, |\nabla^\Sigma \varphi|^2 \, d\gamma - \int_\Sigma \varphi \, \Big | H + \frac{1}{2} \, x^\perp \Big |^2 \, d\gamma \\ 
&\leq \Big ( \int_\Sigma \varphi \, d\gamma \Big ) \, \log \Big ( \int_\Sigma \varphi \, d\gamma \Big ), 
\end{align*} 
where $H$ denotes the mean curvature vector of $\Sigma$.
\end{corollary}

Corollary \ref{log.Sobolev.inequality.alternative.version} can be viewed as a sharp version of Ecker's logarithmic Sobolev inequality for submanifolds \cite{Ecker}. 

In the remainder of this section, we explain how Theorem \ref{log.Sobolev.inequality} can be proved using the Alexandrov-Bakelman-Pucci technique. 

We first consider the special case that $\Sigma$ is connected. By multiplying $f$ by a suitable constant, we may arrange that 
\begin{equation} 
\label{log.Sobolev.normalization}
\int_\Sigma f \, \log f - \int_\Sigma f^{-1} \, |\nabla^\Sigma f|^2 - \int_\Sigma f \, |H|^2 = 0. 
\end{equation}
Since $\Sigma$ is connected, we can find a smooth function $u: \Sigma \to \mathbb{R}$ such that 
\begin{equation}
\label{log.Sobolev.pde.for.u} 
\text{\rm div}_\Sigma(f \, \nabla^\Sigma u) = f \, \log f - f^{-1} \, |\nabla^\Sigma f|^2 - f \, |H|^2. 
\end{equation}
We define 
\begin{align*} 
U &:= \{(x,y): x \in \Sigma, \, y \in T_x^\perp \Sigma\}, \\ 
A &:= \{(x,y) \in U: D_\Sigma^2 u(x) - \langle I\!I(x),y \rangle \geq 0\}. 
\end{align*} 
Moreover, we define a map $\Phi: U \to \mathbb{R}^{n+m}$ by $\Phi(x,y) = \nabla^\Sigma u(x) + y$ for all $(x,y) \in U$. Note that $|\Phi(x,y)|^2 = |\nabla^\Sigma u(x)|^2+|y|^2$ for all $(x,y) \in U$. Arguing as in the proof of Lemma \ref{Michael.Simon.surjectivity}, we can show that 
\[\Phi(A) = \mathbb{R}^{n+m}.\] 
Moreover, the Jacobian determinant of $\Phi$ is given by 
\begin{equation} 
\label{Jacobian.determinant}
\det D\Phi(x,y) = \det (D_\Sigma^2 u(x) - \langle I\!I(x),y \rangle) 
\end{equation}
for each point $(x,y) \in U$. 

\begin{lemma}
\label{log.Sobolev.Jacobian.estimate}
The Jacobian determinant of $\Phi$ satisfies 
\[0 \leq e^{-\frac{1}{4} \, |\Phi(x,y)|^2} \, \det D\Phi(x,y) \leq f(x) \, e^{-\frac{1}{4} \, |2H(x)+y|^2 - n}\] 
for each point $(x,y) \in A$.
\end{lemma}

\textbf{Proof.} 
Let us fix a point $(x,y) \in A$. Using (\ref{log.Sobolev.pde.for.u}), we obtain 
\begin{align*} 
&\Delta_\Sigma u(x) - \langle H(x),y \rangle \\ 
&= \log f(x) - f(x)^{-2} \, |\nabla^\Sigma f(x)|^2 - |H(x)|^2 \\ 
&- f(x)^{-1} \, \langle \nabla^\Sigma f(x),\nabla^\Sigma u(x) \rangle - \langle H(x),y \rangle \\ 
&= \log f(x) + \frac{1}{4} \, (|\nabla^\Sigma u(x)|^2+|y|^2) \\ 
&- \frac{1}{4} \, f(x)^{-2} \, |2 \, \nabla^\Sigma f(x) + f(x) \, \nabla^\Sigma u(x)|^2 - \frac{1}{4} \, |2H(x)+y|^2 \\ 
&\leq \log f(x) + \frac{1}{4} \, (|\nabla^\Sigma u(x)|^2+|y|^2) - \frac{1}{4} \, |2 H(x)+y|^2.
\end{align*}
Moreover, $D_\Sigma^2 u(x) - \langle I\!I(x),y \rangle \geq 0$ since $(x,y) \in A$. Using the elementary inequality $\lambda \leq e^{\lambda-1}$, we obtain 
\begin{align*} 
0 &\leq \det (D_\Sigma^2 u(x) - \langle I\!I(x),y \rangle) \\ 
&\leq e^{\text{\rm tr}(D_\Sigma^2 u(x) - \langle I\!I(x),y \rangle) - n} \\ 
&\leq f(x) \, e^{\frac{1}{4} \, (|\nabla^\Sigma u(x)|^2+|y|^2) - \frac{1}{4} \, |2H(x)+y|^2 - n}.
\end{align*}
Using (\ref{Jacobian.determinant}), we conclude that 
\[0 \leq \det D\Phi(x,y) \leq f(x) \, e^{\frac{1}{4} \, |\Phi(x,y)|^2 - \frac{1}{4} \, |2H(x)+y|^2 - n}.\] 
This completes the proof of Lemma \ref{log.Sobolev.Jacobian.estimate}. \\

Using Lemma \ref{log.Sobolev.Jacobian.estimate} and the fact that $\Phi(A) = \mathbb{R}^{n+m}$, we obtain 
\begin{align*} 
1 &= (4\pi)^{-\frac{n+m}{2}} \int_{\mathbb{R}^{n+m}} e^{-\frac{1}{4} \, |\xi|^2} \, d\xi \\ 
&\leq (4\pi)^{-\frac{n+m}{2}} \int_\Sigma \bigg ( \int_{T_x^\perp \Sigma} e^{-\frac{1}{4} \, |\Phi(x,y)|^2} \, |\det D\Phi(x,y)| \, 1_A(x,y) \, dy \bigg ) \, d\text{\rm vol}(x) \\ 
&\leq (4\pi)^{-\frac{n+m}{2}} \int_\Sigma \bigg ( \int_{T_x^\perp \Sigma} f(x) \, e^{-\frac{1}{4} \, |2H(x)+y|^2 - n} \, dy \bigg ) \, d\text{\rm vol}(x) \\ 
&= (4\pi)^{-\frac{n}{2}} \, e^{-n} \int_\Sigma f(x) \, d\text{\rm vol}(x).
\end{align*}
Rearranging terms gives 
\[n + \frac{n}{2} \, \log(4\pi) \leq \log \Big ( \int_\Sigma f \Big ).\]
Using this estimate together with the normalization condition (\ref{log.Sobolev.normalization}), we conclude that 
\begin{align*} 
&\int_\Sigma f \, \Big ( \log f + n + \frac{n}{2} \, \log(4\pi) \Big ) - \int_\Sigma f^{-1} \, |\nabla^\Sigma f|^2 - \int_\Sigma f \, |H|^2 \\ 
&= \int_\Sigma f \, \Big ( n+\frac{n}{2} \, \log(4\pi) \Big ) \\ 
&\leq \Big ( \int_\Sigma f \Big ) \, \log \Big ( \int_\Sigma f \Big ). 
\end{align*} 
This completes the proof of Theorem \ref{log.Sobolev.inequality} in the special case when $\Sigma$ is connected. 

Finally, if $\Sigma$ is disconnected, we apply the inequality to each connected component of $\Sigma$, and take the sum over all connected components. \\

\section{The Sobolev inequality for manifolds with nonnegative Ricci curvature} 

In this section, we discuss how the isoperimetric inequality in Euclidean space can be generalized to manifolds with nonnegative Ricci curvature. Let $(M,g)$ be a complete noncompact manifold of dimension $n$ with nonnegative Ricci curvature, and let $q$ be an arbitrary point in $M$. It follows from the Bishop-Gromov relative volume comparison theorem that the function 
\[r \mapsto \frac{|\{p \in M: d(p,q) < r\}|}{r^n}\] 
is monotone decreasing. In particular, we can find a real number $\theta \in [0,1]$ such that 
\[\frac{|\{p \in M: d(p,q) < r\}|}{r^n} \to |B^n| \, \theta\] 
as $r \to \infty$. It is easy to see that $\theta$ is independent of the choice of the point $q$. We refer to $\theta$ as the asymptotic volume ratio of $(M,g)$. 

\begin{theorem}[cf. \cite{Agostiniani-Fogagnolo-Mazzieri}, \cite{Brendle3}, \cite{Cordero-Erausquin-McCann-Schmuckenschlager}]
\label{Riemannian.isoperimetric.inequality}
Let $(M,g)$ be a complete noncompact manifold of dimension $n$ with nonnegative Ricci curvature. Let $D$ be a compact domain in $M$ with boundary $\partial D$. Then $|\partial D| \geq n \, |B^n|^{\frac{1}{n}} \, \theta^{\frac{1}{n}} \, |D|^{\frac{n-1}{n}}$, where $\theta$ denotes the asymptotic volume ratio of $(M,g)$.
\end{theorem}

In the three-dimensional case, Theorem \ref{Riemannian.isoperimetric.inequality} was proved by Agostiniani, Fogagnolo, and Mazzieri (cf. \cite{Agostiniani-Fogagnolo-Mazzieri}, Theorem 6.1). Their proof is based on an argument due to Huisken \cite{Huisken} and uses mean curvature flow. In \cite{Brendle3}, we gave a proof of Theorem \ref{Riemannian.isoperimetric.inequality} in all dimensions based on the Alexandrov-Bakelman-Pucci technique. Finally, Balogh and Krist\'aly \cite{Balogh-Kristaly} observed that Theorem \ref{Riemannian.isoperimetric.inequality} can alternatively be deduced from the Riemannian interpolation inequality of Cordero-Erausquin, McCann, and Schmuckenschl\"ager \cite{Cordero-Erausquin-McCann-Schmuckenschlager}. 

Theorem \ref{Riemannian.isoperimetric.inequality} is a special case of the following Sobolev inequality. 

\begin{theorem}
\label{Riemannian.Sobolev.inequality}
Let $(M,g)$ be a complete noncompact manifold of dimension $n$ with nonnegative Ricci curvature. Let $D$ be a compact domain in $M$ with boundary $\partial D$, and let $f$ be a positive smooth function on $D$. Then 
\[\int_D |\nabla f| + \int_{\partial D} f \geq n \, |B^n|^{\frac{1}{n}} \, \theta^{\frac{1}{n}} \, \Big ( \int_D f^{\frac{n}{n-1}} \Big )^{\frac{n-1}{n}},\] 
where $\theta$ denotes the asymptotic volume ratio of $(M,g)$.
\end{theorem}

The Michael-Simon Sobolev inequality can also be generalized to the Riemannian setting, assuming that the ambient manifold has nonnegative sectional curvature (cf. \cite{Brendle3}).

In the remainder of this section, we describe the proof of Theorem \ref{Riemannian.Sobolev.inequality} based on the Alexandrov-Bakelman-Pucci maximum principle. We will follow the arguments in \cite{Brendle3}. 

We first consider the special case that $D$ is connected. By multiplying $f$ by a suitable constant, we may arrange that 
\begin{equation} 
\label{Riemannian.Sobolev.normalization}
\int_D |\nabla f| + \int_{\partial D} f = n \int_D f^{\frac{n}{n-1}}. 
\end{equation}
Since $D$ is connected, we can find a function $u: D \to \mathbb{R}$ such that 
\begin{equation} 
\label{Riemannian.Sobolev.pde.for.u}
\text{\rm div}(f \, \nabla u) = n \, f^{\frac{n}{n-1}} - |\nabla f| 
\end{equation}
at each point in $D$ and 
\begin{equation} 
\label{Riemannian.Sobolev.boundary.condition.for.u}
\langle \nabla u,\eta \rangle = 1 
\end{equation}
at each point on $\partial D$, where $\eta$ denotes the outward-pointing unit normal to $\partial D$. The function $u$ is of class $C^{2,\gamma}$ for each $0 < \gamma < 1$. We define 
\[U := \{x \in D \setminus \partial D: |\nabla u(x)| < 1\}.\] 
For each $r > 0$, we denote by $A_r$ the set of all points $\bar{x} \in U$ with the property that 
\[r \, u(x) + \frac{1}{2} \, d \big ( x,\exp_{\bar{x}}(r \, \nabla u(\bar{x})) \big )^2 \geq r \, u(\bar{x}) + \frac{1}{2} \, r^2 \, |\nabla u(\bar{x})|^2\] 
for all $x \in D$. Moreover, for each $r > 0$, we define a map $\Phi_r: U \to M$ by 
\[\Phi_r(x) = \exp_x(r \, \nabla u(x))\] 
for all $x \in U$.

\begin{lemma} 
\label{Laplacian}
Assume that $x \in U$. Then $\Delta u(x) \leq n \, f(x)^{\frac{1}{n-1}}$. 
\end{lemma}

\textbf{Proof.} 
Let us fix a point $x \in U$. Using (\ref{Riemannian.Sobolev.pde.for.u}), we obtain 
\begin{align*} 
\Delta u(x) 
&= n \, f(x)^{\frac{1}{n-1}} - f(x)^{-1} \, |\nabla f(x)| - f(x)^{-1} \, \langle \nabla f(x),\nabla u(x) \rangle \\ 
&\leq n \, f(x)^{\frac{1}{n-1}}. 
\end{align*}
This completes the proof of Lemma \ref{Laplacian}. \\

\begin{lemma} 
\label{Riemannian.Sobolev.surjectivity}
For each $r>0$, the set 
\[\Big \{ p \in M: \sup_{x \in D} d(x,p) < r \Big \}\] 
is contained in the image $\Phi_r(A_r)$. 
\end{lemma}

\textbf{Proof.} 
Let us fix a real number $r>0$. Let $p$ be a point in $M$ with the property that $\sup_{x \in D} d(x,p) < r$. We can find a point $\bar{x} \in D$ such that 
\[r \, u(\bar{x}) + \frac{1}{2} \, d(\bar{x},p)^2 = \inf_{x \in D} (r \, u(x) + \frac{1}{2} \, d(x,p)^2).\]
Using (\ref{Riemannian.Sobolev.boundary.condition.for.u}) and the inequality $\sup_{x \in D} d(x,p) < r$, we conclude that the point $\bar{x}$ lies in the interior of $D$. Let $\bar{\gamma}: [0,r] \to M$ be a minimizing geodesic such that $\bar{\gamma}(0) = \bar{x}$ and $\bar{\gamma}(r) = p$. Clearly, $r \, |\bar{\gamma}'(t)| = d(\bar{x},p)$ for all $t \in [0,r]$. Moreover, 
\begin{align*} 
r \, u(\gamma(0)) + \frac{1}{2} \, r \int_0^r |\gamma'(t)|^2 \, dt 
&\geq r \, u(\gamma(0)) + \frac{1}{2} \, d(\gamma(0),p)^2 \\ 
&\geq r \, u(\bar{x}) + \frac{1}{2} \, d(\bar{x},p)^2 
\end{align*}
for every smooth path $\gamma: [0,r] \to M$ satisfying $\gamma(0) \in D$ and $\gamma(r) = p$. Therefore, the path $\bar{\gamma}$ minimizes the functional $u(\gamma(0)) + \frac{1}{2} \int_0^r |\gamma'(t)|^2 \, dt$ among all smooth paths $\gamma: [0,r] \to M$ satisfying $\gamma(0) \in D$ and $\gamma(r) = p$. The first variation formula implies that 
\[\nabla u(\bar{x}) = \bar{\gamma}'(0).\] 
Consequently, 
\[\exp_{\bar{x}}(r \, \nabla u(\bar{x})) = \exp_{\bar{\gamma}(0)}(r \, \bar{\gamma}'(0)) = \bar{\gamma}(r) = p.\] 
Moreover, 
\[r \, |\nabla u(\bar{x})| = r \, |\bar{\gamma}'(0)| = d(\bar{x},p).\] 
Since $d(\bar{x},p) < r$, it follows that $|\nabla u(\bar{x})| < 1$. Therefore, $\bar{x} \in U$. Finally, 
\begin{align*} 
r \, u(x) + \frac{1}{2} \, d \big ( x,\exp_{\bar{x}}(r \, \nabla u(\bar{x})) \big )^2 
&= r \, u(x) + \frac{1}{2} \, d(x,p)^2 \\ 
&\geq r \, u(\bar{x}) + \frac{1}{2} \, d(\bar{x},p)^2 \\ 
&= r \, u(\bar{x}) + \frac{1}{2} \, r^2 \, |\nabla u(\bar{x})|^2 
\end{align*} 
for each point $x \in D$. Thus, $\bar{x} \in A_r$ and $\Phi_r(\bar{x}) = p$. This completes the proof of Lemma \ref{Riemannian.Sobolev.surjectivity}. \\

\begin{lemma} 
\label{Riemannian.Sobolev.second.variation}
Assume that $r>0$ and $\bar{x} \in A_r$. We define $\bar{\gamma}(t) := \exp_{\bar{x}}(t \, \nabla u(\bar{x}))$ for all $t \in [0,r]$. If $Z$ is a smooth vector field along $\bar{\gamma}$ satisfying $Z(r) = 0$, then 
\[(D^2 u)(Z(0),Z(0)) + \int_0^r \big ( |D_t Z(t)|^2 - R(\bar{\gamma}'(t),Z(t),\bar{\gamma}'(t),Z(t)) \big ) \, dt \geq 0.\]
\end{lemma}

\textbf{Proof.} 
Since $\bar{x} \in A_r$, we obtain 
\begin{align*} 
r \, u(\gamma(0)) + \frac{1}{2} \, r \int_0^r |\gamma'(t)|^2 \, dt 
&\geq r \, u(\gamma(0)) + \frac{1}{2} \, d(\gamma(0),\bar{\gamma}(r))^2 \\ 
&\geq r \, u(\bar{x}) + \frac{1}{2} \, r^2 \, |\nabla u(\bar{x})|^2 
\end{align*} 
for every smooth path $\gamma: [0,r] \to M$ satisfying $\gamma(0) \in D$ and $\gamma(r) = \bar{\gamma}(r)$. Therefore, the path $\bar{\gamma}$ minimizes the functional $u(\gamma(0)) + \frac{1}{2} \int_0^r |\gamma'(t)|^2 \, dt$ among all smooth paths $\gamma: [0,r] \to M$ satisfying $\gamma(0) \in D$ and $\gamma(r) = \bar{\gamma}(r)$. Hence, the assertion follows from the second variation formula. This completes the proof of Lemma \ref{Riemannian.Sobolev.second.variation}. \\

\begin{lemma} 
\label{Riemannian.Sobolev.no.conjugate.points}
Assume that $r>0$ and $\bar{x} \in A_r$. We define $\bar{\gamma}(t) := \exp_{\bar{x}}(t \, \nabla u(\bar{x}))$ for all $t \in [0,r]$. Suppose that $W(t)$ is a Jacobi field along $\bar{\gamma}$ with the property that $D_t W(0) = \sum_{j=1}^n (D^2 u)(W(0),e_j) \, e_j$, where $\{e_1,\hdots,e_n\}$ denotes an orthonormal basis of $T_{\bar{x}} M$. If $W(\tau) = 0$ for some $\tau \in (0,r)$, then $W$ vanishes identically.
\end{lemma}

\textbf{Proof.} 
We argue by contradiction. Suppose that $W(\tau) = 0$ for some $\tau \in (0,r)$, but $W$ does not vanish identically. Using standard uniqueness results for ODE, we conclude that $D_t W(\tau) \neq 0$. Since $W$ is a Jacobi field, we obtain 
\begin{align*} 
&\int_0^\tau \big ( |D_t W(t)|^2 - R(\bar{\gamma}'(t),W(t),\bar{\gamma}'(t),W(t)) \big ) \, dt \\ 
&= \langle D_t W(\tau),W(\tau) \rangle - \langle D_t W(0),W(0) \rangle \\ 
&= -(D^2 u)(W(0),W(0)). 
\end{align*}
We define a vector field $\tilde{W}$ along $\bar{\gamma}$ by 
\[\tilde{W}(t) = \begin{cases} W(t) & \text{\rm for $t \in [0,\tau]$} \\ 0 & \text{\rm for $t \in [\tau,r]$.} \end{cases}\] 
Clearly, $\tilde{W}(r) = 0$ and 
\begin{align*} 
&\int_0^r \big ( |D_t \tilde{W}(t)|^2 - R(\bar{\gamma}'(t),\tilde{W}(t),\bar{\gamma}'(t),\tilde{W}(t)) \big ) \, dt \\ 
&= -(D^2 u)(\tilde{W}(0),\tilde{W}(0)). 
\end{align*} 
Since $D_t W(\tau) \neq 0$, the covariant derivative of $\tilde{W}(t)$ is discontinuous at the point $t = \tau$. Consequently, we can find a smooth vector field $Z$ along $\bar{\gamma}$ such that $Z(0) = \tilde{W}(0)$, $Z(r) = \tilde{W}(r)$, and 
\begin{align*} 
&\int_0^r \big ( |D_t Z(t)|^2 - R(\bar{\gamma}'(t),Z(t),\bar{\gamma}'(t),Z(t)) \big ) \, dt \\ 
&< -(D^2 u)(Z(0),Z(0)). 
\end{align*}
This contradicts Lemma \ref{Riemannian.Sobolev.second.variation}. This completes the proof of Lemma \ref{Riemannian.Sobolev.no.conjugate.points}. \\

\begin{proposition} 
\label{Riemannian.Sobolev.monotonicity}
Assume that $r>0$ and $x \in A_r$. Then 
\[\lim_{t \to 0} |\det D\Phi_t(x)| = 1.\] 
Moreover, the function 
\[t \mapsto (1+t \, f(x)^{\frac{1}{n-1}})^{-n} \, |\det D\Phi_t(x)|\] 
is monotone decreasing for $t \in (0,r)$. 
\end{proposition}

\textbf{Proof.} 
Let us fix a real number $r>0$ and a point $\bar{x} \in A_r$. We define $\bar{\gamma}(t) := \exp_{\bar{x}}(t \, \nabla u(\bar{x}))$ for all $t \in [0,r]$. Let $\{e_1,\hdots,e_n\}$ be an orthonormal basis of $T_{\bar{x}} M$. For each $1 \leq i \leq n$, we denote by $E_i(t)$ the unique parallel vector field along $\bar{\gamma}$ satisfying $E_i(0) = e_i$. Moreover, for each $1 \leq i \leq n$, we denote by $X_i(t)$ the unique Jacobi field along $\bar{\gamma}$ satisfying $X_i(0) = e_i$ and 
\[D_t X_i(0) = \sum_{j=1}^n (D^2 u)(e_i,e_j) \, e_j.\] 
It follows from Lemma \ref{Riemannian.Sobolev.no.conjugate.points} that $X_1(t),\hdots,X_n(t)$ are linearly independent for each $t \in (0,r)$. 

For each $t \in [0,r]$, we define an $n \times n$-matrix $P(t)$ by 
\[P_{ij}(t) = \langle X_i(t),E_j(t) \rangle\] 
for $1 \leq i,j \leq n$. Moreover, for each $t \in [0,r]$, we define an $n \times n$-matrix $S(t)$ by 
\[S_{ij}(t) = R(\bar{\gamma}'(t),E_i(t),\bar{\gamma}'(t),E_j(t))\] 
for $1 \leq i,j \leq n$. The matrix $S(t)$ is symmetric for each $t \in [0,r]$. Moreover, 
\[\text{\rm tr}(S(t)) = \text{\rm Ric}(\bar{\gamma}'(t),\bar{\gamma}'(t)) \geq 0\] 
for each $t \in [0,r]$. Since the vector fields $X_1(t),\hdots,X_n(t)$ are Jacobi fields, we obtain 
\begin{align*} 
P_{ij}''(t) 
&= \langle D_t D_t X_i(t),E_j(t) \rangle \\ 
&= -R(\bar{\gamma}'(t),X_i(t),\bar{\gamma}'(t),E_j(t)) \\ 
&= -\sum_{k=1}^n P_{ik}(t) \, S_{kj}(t) 
\end{align*}
for all $1 \leq i,j \leq n$ and all $t \in [0,r]$. In other words, the function $P(t)$ satisfies 
\[P''(t) = -P(t) S(t)\] 
for all $t \in [0,r]$. Moreover, 
\[P_{ij}(0) = \langle X_i(0),E_j(0) \rangle = \delta_{ij}\] 
and 
\[P_{ij}'(0) = \langle D_t X_i(0),E_j(0) \rangle = (D^2 u)(e_i,e_j)\] 
for $1 \leq i,j \leq n$. In particular, the matrix $P'(0) P(0)^T$ is symmetric. Moreover, since $S(t)$ is symmetric for each $t \in [0,r]$, it follows that the matrix 
\begin{align*} 
\frac{d}{dt} (P'(t) P(t)^T) 
&= P''(t) P(t)^T + P'(t) P'(t)^T \\ 
&= -P(t) S(t) P(t)^T + P'(t) P'(t)^T 
\end{align*}
is symmetric for each $t \in [0,r]$. Consequently, the matrix $P'(t) P(t)^T$ is symmetric for each $t \in [0,r]$. 

Since $X_1(t),\hdots,X_n(t)$ are linearly independent for each $t \in (0,r)$, the matrix $P(t)$ is invertible for each $t \in (0,r)$. Since the matrix $P'(t) P(t)^T$ is symmetric for each $t \in (0,r)$, it follows that the matrix $Q(t) := P(t)^{-1} P'(t)$ is symmetric for each $t \in (0,r)$. The function $Q(t)$ satisfies the Riccati equation 
\[Q'(t) = P(t)^{-1} P''(t) - P(t)^{-1} P'(t) P(t)^{-1} P'(t) = -S(t) - Q(t)^2\] 
for all $t \in (0,r)$. Moreover, since $Q(t)$ is symmetric, we obtain $\text{\rm tr}(Q(t)^2) \geq \frac{1}{n} \, \text{\rm tr}(Q(t))^2$ for all $t \in (0,r)$. Since $\text{\rm tr}(S(t)) \geq 0$ for all $t \in (0,r)$, it follows that 
\[\frac{d}{dt} \text{\rm tr}(Q(t)) = -\text{\rm tr}(S(t)) - \text{\rm tr}(Q(t)^2) \leq -\frac{1}{n} \, \text{\rm tr}(Q(t))^2\] 
for all $t \in (0,r)$. Using Lemma \ref{Laplacian}, we obtain 
\[\lim_{t \to 0} \text{\rm tr}(Q(t)) = \Delta u(\bar{x}) \leq n \, f(\bar{x})^{\frac{1}{n-1}}.\] 
A standard ODE comparison principle gives 
\[\text{\rm tr}(Q(t)) \leq n \, \big ( 1+t \, f(\bar{x})^{\frac{1}{n-1}} \big )^{-1} \, f(\bar{x})^{\frac{1}{n-1}}\] 
for all $t \in (0,r)$.

Since $P(0) = \text{\rm id}$, we know that $\det P(t) > 0$ if $t > 0$ is sufficiently small. Since $P(t)$ is invertible for each $t \in (0,r)$, it follows that $\det P(t) > 0$ for each $t \in (0,r)$. Moreover, 
\[\frac{d}{dt} \log \det P(t) = \text{\rm tr}(Q(t)) \leq n \, \big ( 1+t \, f(\bar{x})^{\frac{1}{n-1}} \big )^{-1} \, f(\bar{x})^{\frac{1}{n-1}}\] 
for all $t \in (0,r)$. Consequently, the function 
\[t \mapsto (1+t \, f(\bar{x})^{\frac{1}{n-1}})^{-n} \, \det P(t)\] 
is monotone decreasing for $t \in (0,r)$. 

Finally, the differential $(D\Phi_t)_{\bar{x}}: T_{\bar{x}} M \to T_{\bar{\gamma}(t)} M$ maps the vector $e_i \in T_{\bar{x}} M$ to the vector $X_i(t) \in T_{\bar{\gamma}(t)} M$. Consequently, $|\det D\Phi_t(\bar{x})| = \det P(t)$ for all $t \in (0,r)$. Putting these facts together, the assertion follows. This completes the proof of Proposition \ref{Riemannian.Sobolev.monotonicity}. \\

\begin{corollary}
\label{Riemannian.Sobolev.upper.bound.for.Jacobian}
The Jacobian determinant of $\Phi_r$ satisfies 
\[|\det D\Phi_r(x)| \leq (1+r \, f(x)^{\frac{1}{n-1}})^n\] 
for each $r>0$ and each point $x \in A_r$. 
\end{corollary}

We now continue the proof of Theorem \ref{Riemannian.Sobolev.inequality}. Using Lemma \ref{Riemannian.Sobolev.surjectivity} and Corollary \ref{Riemannian.Sobolev.upper.bound.for.Jacobian}, we obtain 
\begin{align*} 
\Big | \Big \{ p \in M: \sup_{x \in D} d(x,p) < r \Big \} \Big | 
&\leq \int_{A_r} |\det D\Phi_r(x)| \, d\text{\rm vol}(x) \\ 
&\leq \int_D (1+r \, f(x)^{\frac{1}{n-1}})^n \, d\text{\rm vol}(x)
\end{align*}
for all $r > 0$. Finally, we divide by $r^n$ and send $r \to \infty$. This gives 
\[|B^n| \, \theta \leq \int_D f^{\frac{n}{n-1}}.\] 
Using the normalization condition (\ref{Riemannian.Sobolev.normalization}), we conclude that 
\[\int_D |\nabla f| + \int_{\partial D} f = n \int_D f^{\frac{n}{n-1}} \geq n \, |B^n|^{\frac{1}{n}} \, \theta^{\frac{1}{n}} \, \Big ( \int_D f^{\frac{n}{n-1}} \Big )^{\frac{n-1}{n}}.\]
This completes the proof of Theorem \ref{Riemannian.Sobolev.inequality} in the special case when $D$ is connected. 

Finally, if $D$ is disconnected, we apply the inequality to each connected component of $D$, and take the sum over all connected components. \\

\section{The Fenchel-Willmore-Chen inequality for hypersurfaces in manifolds with nonnegative Ricci curvature}

In this final section, we discuss a version of the Fenchel-Willmore-Chen inequality for hypersurfaces in manifolds with nonnegative Ricci curvature.

\begin{theorem}[cf. \cite{Agostiniani-Fogagnolo-Mazzieri}, \cite{Heintze-Karcher}]
\label{Riemannian.FWC.estimate}
Let $(M,g)$ be a complete noncompact manifold of dimension $n$ with nonnegative Ricci curvature. Let $\Sigma$ be a compact hypersurface in $M$ without boundary. Then 
\[\int_\Sigma \Big ( \frac{|H|}{n-1} \Big )^{n-1} \geq |S^{n-1}| \, \theta,\] 
where $H$ denotes the mean curvature vector of $\Sigma$ and $\theta$ denotes the asymptotic volume ratio of $(M,g)$.
\end{theorem}

Theorem \ref{Riemannian.FWC.estimate} is a consequence of an important estimate due to Heintze and Karcher \cite{Heintze-Karcher}. Specifically, Theorem 2.1 in \cite{Heintze-Karcher} gives an upper bound for the volume of a tubular neighborhood of $\Sigma$ of radius $r$. Theorem \ref{Riemannian.FWC.estimate} follows from this inequality by sending $r \to \infty$; see the comment at the top of p.~452 in \cite{Heintze-Karcher}. Agostiniani, Fogagnolo, and Mazzieri \cite{Agostiniani-Fogagnolo-Mazzieri} recently gave an alternative proof of Theorem \ref{Riemannian.FWC.estimate}.

Theorem \ref{Riemannian.FWC.estimate} can be generalized to higher codimension, provided that the ambient manifold has nonnegative sectional curvature. This again follows from Theorem 2.1 in \cite{Heintze-Karcher}.

In the remainder of this section, we discuss the proof of Theorem \ref{Riemannian.FWC.estimate}, following the work of Heintze and Karcher \cite{Heintze-Karcher}.

We define 
\[U := \{(x,y): x \in \Sigma, \, y \in T_x^\perp \Sigma, \, |y| < 1\}.\] 
For each $r > 0$, we denote by $A_r$ the set of all points $(\bar{x},\bar{y}) \in U$ with the property that 
\[d \big ( x,\exp_{\bar{x}}(r\bar{y}) \big ) \geq r \, |\bar{y}|\] 
for all $x \in \Sigma$. Moreover, for each $r > 0$, we define a map $\Phi_r: U \to M$ by 
\[\Phi_r(x,y) = \exp_x(ry)\] 
for all $(x,y) \in U$.

\begin{lemma} 
\label{Riemannian.FWC.surjectivity}
For each $r>0$, the set 
\[\Big \{ p \in M: \inf_{x \in \Sigma} d(x,p) < r \Big \}\] 
is contained in the image $\Phi_r(A_r)$. 
\end{lemma}

\textbf{Proof.} 
Let us fix a real number $r>0$. Let $p$ be a point in $M$ with the property that $\inf_{x \in \Sigma} d(x,p) < r$. We can find a point $\bar{x} \in \Sigma$ such that 
\[d(\bar{x},p) = \inf_{x \in \Sigma} d(x,p).\] 
Then $d(\bar{x},p) < r$. Let $\bar{\gamma}: [0,r] \to M$ be a minimizing geodesic such that $\bar{\gamma}(0) = \bar{x}$ and $\bar{\gamma}(r) = p$. Clearly, $r \, |\bar{\gamma}'(t)| = d(\bar{x},p)$ for all $t \in [0,r]$. Moreover, 
\[r \int_0^r |\gamma'(t)|^2 \, dt \geq d(\gamma(0),p)^2 \geq d(\bar{x},p)^2\] 
for every smooth path $\gamma: [0,r] \to M$ satisfying $\gamma(0) \in \Sigma$ and $\gamma(r) = p$. Therefore, the path $\bar{\gamma}$ minimizes the functional $\int_0^r |\gamma'(t)|^2 \, dt$ among all smooth paths $\gamma: [0,r] \to M$ satisfying $\gamma(0) \in \Sigma$ and $\gamma(r) = p$. The first variation formula implies that $\bar{\gamma}'(0) \in T_{\bar{x}}^\perp \Sigma$. We define $\bar{y} = \bar{\gamma}'(0)$. Then 
\[\exp_{\bar{x}}(r\bar{y}) = \exp_{\bar{\gamma}(0)}(r \, \bar{\gamma}'(0)) = \bar{\gamma}(r) = p.\]
Moreover, 
\[r \, |\bar{y}| = r \, |\bar{\gamma}'(0)| = d(\bar{x},p).\] 
Since $d(\bar{x},p) < r$, it follows that $|\bar{y}| < 1$. Therefore, $(\bar{x},\bar{y}) \in U$. Finally, 
\[d \big ( x,\exp_{\bar{x}}(r\bar{y}) \big ) = d(x,p) \geq d(\bar{x},p) = r \, |\bar{y}|\] 
for each point $x \in \Sigma$. Thus, $(\bar{x},\bar{y}) \in A_r$ and $\Phi_r(\bar{x},\bar{y}) = p$. This completes the proof of Lemma \ref{Riemannian.FWC.surjectivity}. \\

\begin{lemma} 
\label{Riemannian.FWC.second.variation}
Assume that $r>0$ and $(\bar{x},\bar{y}) \in A_r$. We define $\bar{\gamma}(t) := \exp_{\bar{x}}(t \bar{y})$ for all $t \in [0,r]$. If $Z$ is a smooth vector field along $\bar{\gamma}$ such that $Z(0) \in T_{\bar{x}} \Sigma$ and $Z(r) = 0$, then 
\[-\langle I\!I(Z(0),Z(0)),\bar{y} \rangle + \int_0^r \big ( |D_t Z(t)|^2 - R(\bar{\gamma}'(t),Z(t),\bar{\gamma}'(t),Z(t)) \big ) \, dt \geq 0.\]
\end{lemma}

\textbf{Proof.} 
Since $(\bar{x},\bar{y}) \in A_r$, we obtain 
\[r \int_0^r |\gamma'(t)|^2 \, dt \geq d(\gamma(0),\bar{\gamma}(r))^2 \geq r^2 \, |\bar{y}|^2\] 
for every smooth path $\gamma: [0,r] \to M$ satisfying $\gamma(0) \in \Sigma$ and $\gamma(r) = \bar{\gamma}(r)$. Therefore, the path $\bar{\gamma}$ minimizes the functional $\int_0^r |\gamma'(t)|^2 \, dt$ among all smooth paths $\gamma: [0,r] \to M$ satisfying $\gamma(0) \in \Sigma$ and $\gamma(r) = \bar{\gamma}(r)$. Hence, the assertion follows from the second variation formula. This completes the proof of Lemma \ref{Riemannian.FWC.second.variation}. \\

\begin{lemma} 
\label{Riemannian.FWC.no.conjugate.points}
Assume that $r>0$ and $(\bar{x},\bar{y}) \in A_r$. We define $\bar{\gamma}(t) := \exp_{\bar{x}}(t \bar{y})$ for all $t \in [0,r]$. Suppose that $W(t)$ is a Jacobi field along $\bar{\gamma}$ with the property that $W(0) \in T_{\bar{x}} \Sigma$ and $D_t W(0) = -\sum_{j=1}^{n-1} \langle I\!I(W(0),e_j),\bar{y} \rangle \, e_j$, where $\{e_1,\hdots,e_{n-1}\}$ denotes an orthonormal basis of $T_{\bar{x}} \Sigma$. If $W(\tau) = 0$ for some $\tau \in (0,r)$, then $W$ vanishes identically.
\end{lemma}

\textbf{Proof.} 
We argue by contradiction. Suppose that $W(\tau) = 0$ for some $\tau \in (0,r)$, but $W$ does not vanish identically. Using standard uniqueness results for ODE, we conclude that $D_t W(\tau) \neq 0$. Since $W$ is a Jacobi field, we obtain 
\begin{align*} 
&\int_0^\tau \big ( |D_t W(t)|^2 - R(\bar{\gamma}'(t),W(t),\bar{\gamma}'(t),W(t)) \big ) \, dt \\ 
&= \langle D_t W(\tau),W(\tau) \rangle - \langle D_t W(0),W(0) \rangle \\ 
&= \langle I\!I(W(0),W(0)),\bar{y} \rangle. 
\end{align*}
We define a vector field $\tilde{W}$ along $\bar{\gamma}$ by 
\[\tilde{W}(t) = \begin{cases} W(t) & \text{\rm for $t \in [0,\tau]$} \\ 0 & \text{\rm for $t \in [\tau,r]$.} \end{cases}\] 
Clearly, $\tilde{W}(r) = 0$ and 
\begin{align*} 
&\int_0^r \big ( |D_t \tilde{W}(t)|^2 - R(\bar{\gamma}'(t),\tilde{W}(t),\bar{\gamma}'(t),\tilde{W}(t)) \big ) \, dt \\ 
&= \langle I\!I(\tilde{W}(0),\tilde{W}(0)),\bar{y} \rangle. 
\end{align*} 
Since $D_t W(\tau) \neq 0$, the covariant derivative of $\tilde{W}(t)$ is discontinuous at the point $t = \tau$. Consequently, we can find a smooth vector field $Z$ along $\bar{\gamma}$ such that $Z(0) = \tilde{W}(0)$, $Z(r) = \tilde{W}(r)$, and 
\begin{align*} 
&\int_0^r \big ( |D_t Z(t)|^2 - R(\bar{\gamma}'(t),Z(t),\bar{\gamma}'(t),Z(t)) \big ) \, dt \\ 
&< \langle I\!I(Z(0),Z(0)),\bar{y} \rangle. 
\end{align*}
This contradicts Lemma \ref{Riemannian.FWC.second.variation}. This completes the proof of Lemma \ref{Riemannian.FWC.no.conjugate.points}. \\

\begin{proposition} 
\label{Riemannian.FWC.monotonicity}
Assume that $r>0$ and $(x,y) \in A_r$. Then 
\[\lim_{t \to 0} t^{-1} \, |\det D\Phi_t(x,y)| = 1.\] 
Moreover, 
\[1 - r \, \frac{\langle H(x),y \rangle}{n-1} \geq 0\] 
and the function 
\[t \mapsto t^{-1} \, \Big ( 1 - t \, \frac{\langle H(x),y \rangle}{n-1} \Big )^{-(n-1)} \, |\det D\Phi_t(x,y)|\] 
is monotone decreasing for $t \in (0,r)$. 
\end{proposition}

\textbf{Proof.} 
Let us fix a real number $r>0$ and a point $(\bar{x},\bar{y}) \in A_r$. We define $\bar{\gamma}(t) := \exp_{\bar{x}}(t \bar{y})$ for all $t \in [0,r]$. Let $\{e_1,\hdots,e_{n-1}\}$ be an orthonormal basis of $T_{\bar{x}} \Sigma$. For each $1 \leq i \leq n-1$, we denote by $E_i(t)$ the unique parallel vector field along $\bar{\gamma}$ satisfying $E_i(0) = e_i$. Moreover, for each $1 \leq i \leq n-1$, we denote by $X_i(t)$ the unique Jacobi field along $\bar{\gamma}$ satisfying $X_i(0) = e_i$ and 
\[D_t X_i(0) = -\sum_{j=1}^{n-1} \langle I\!I(e_i,e_j),\bar{y} \rangle \, e_j.\] 
For each $1 \leq i \leq n-1$, $X_i(t)$ is a normal Jacobi field along $\bar{\gamma}$. In other words, $X_i(t) \in \text{\rm span}\{E_1(t),\hdots,E_{n-1}(t)\}$ for all $1 \leq i \leq n-1$ and all $t \in [0,r]$. It follows from Lemma \ref{Riemannian.FWC.no.conjugate.points} that $X_1(t),\hdots,X_{n-1}(t)$ are linearly independent for each $t \in (0,r)$. 

For each $t \in [0,r]$, we define an $(n-1) \times (n-1)$-matrix $P(t)$ by 
\[P_{ij}(t) = \langle X_i(t),E_j(t) \rangle\] 
for $1 \leq i,j \leq n-1$. Moreover, for each $t \in [0,r]$, we define an $(n-1) \times (n-1)$-matrix $S(t)$ by 
\[S_{ij}(t) = R(\bar{\gamma}'(t),E_i(t),\bar{\gamma}'(t),E_j(t))\] 
for $1 \leq i,j \leq n-1$. The matrix $S(t)$ is symmetric for each $t \in [0,r]$. Moreover, 
\[\text{\rm tr}(S(t)) = \text{\rm Ric}(\bar{\gamma}'(t),\bar{\gamma}'(t)) \geq 0\] 
for each $t \in [0,r]$. Since the vector fields $X_1(t),\hdots,X_{n-1}(t)$ are Jacobi fields, we obtain 
\begin{align*} 
P_{ij}''(t) 
&= \langle D_t D_t X_i(t),E_j(t) \rangle \\ 
&= -R(\bar{\gamma}'(t),X_i(t),\bar{\gamma}'(t),E_j(t)) \\ 
&= -\sum_{k=1}^{n-1} P_{ik}(t) \, S_{kj}(t) 
\end{align*}
for all $1 \leq i,j \leq n-1$ and all $t \in [0,r]$. In other words, the function $P(t)$ satisfies 
\[P''(t) = -P(t) S(t)\] 
for all $t \in [0,r]$. Moreover, 
\[P_{ij}(0) = \langle X_i(0),E_j(0) \rangle = \delta_{ij}\] 
and 
\[P_{ij}'(0) = \langle D_t X_i(0),E_j(0) \rangle = -\langle I\!I(e_i,e_j),\bar{y} \rangle\] 
for $1 \leq i,j \leq n-1$. In particular, the matrix $P'(0) P(0)^T$ is symmetric. Moreover, since $S(t)$ is symmetric for each $t \in [0,r]$, it follows that the matrix 
\begin{align*} 
\frac{d}{dt} (P'(t) P(t)^T) 
&= P''(t) P(t)^T + P'(t) P'(t)^T \\ 
&= -P(t) S(t) P(t)^T + P'(t) P'(t)^T 
\end{align*}
is symmetric for each $t \in [0,r]$. Consequently, the matrix $P'(t) P(t)^T$ is symmetric for each $t \in [0,r]$. 

Since $X_1(t),\hdots,X_{n-1}(t)$ are linearly independent for each $t \in (0,r)$, the matrix $P(t)$ is invertible for each $t \in (0,r)$. Since the matrix $P'(t) P(t)^T$ is symmetric for each $t \in (0,r)$, it follows that the matrix $Q(t) := P(t)^{-1} P'(t)$ is symmetric for each $t \in (0,r)$. The function $Q(t)$ satisfies the Riccati equation 
\[Q'(t) = P(t)^{-1} P''(t) - P(t)^{-1} P'(t) P(t)^{-1} P'(t) = -S(t) - Q(t)^2\] 
for all $t \in (0,r)$. Moreover, since $Q(t)$ is symmetric, we obtain $\text{\rm tr}(Q(t)^2) \geq \frac{1}{n-1} \, \text{\rm tr}(Q(t))^2$ for all $t \in (0,r)$. Since $\text{\rm tr}(S(t)) \geq 0$ for all $t \in (0,r)$, it follows that 
\[\frac{d}{dt} \text{\rm tr}(Q(t)) = -\text{\rm tr}(S(t)) - \text{\rm tr}(Q(t)^2) \leq -\frac{1}{n-1} \, \text{\rm tr}(Q(t))^2\] 
for all $t \in (0,r)$. Moreover, 
\[\lim_{t \to 0} \text{\rm tr}(Q(t)) = -\langle H(\bar{x}),\bar{y} \rangle.\] 
A standard ODE comparison principle implies that 
\[1 - r \, \frac{\langle H(\bar{x}),\bar{y} \rangle}{n-1} \geq 0\] 
and 
\[\text{\rm tr}(Q(t)) \leq -\Big ( 1 - t \, \frac{\langle H(\bar{x}),\bar{y} \rangle}{n-1} \Big )^{-1} \, \langle H(\bar{x}),\bar{y} \rangle\] 
for all $t \in (0,r)$.

Since $P(0) = \text{\rm id}$, we know that $\det P(t) > 0$ if $t > 0$ is sufficiently small. Since $P(t)$ is invertible for each $t \in (0,r)$, it follows that $\det P(t) > 0$ for each $t \in (0,r)$. Moreover, 
\[\frac{d}{dt} \log \det P(t) = \text{\rm tr}(Q(t)) \leq -\Big ( 1 - t \, \frac{\langle H(\bar{x}),\bar{y} \rangle}{n-1} \Big )^{-1} \, \langle H(\bar{x}),\bar{y} \rangle\] 
for all $t \in (0,r)$. Consequently, the function 
\[t \mapsto \Big ( 1 - t \, \frac{\langle H(\bar{x}),\bar{y} \rangle}{n-1} \Big )^{-(n-1)} \, \det P(t)\] 
is monotone decreasing for $t \in (0,r)$. 

Finally, we compute the differential of $\Phi_t$ at the point $(\bar{x},\bar{y})$. To that end, we decompose the tangent space to $U$ at the point $(\bar{x},\bar{y})$ into the horizontal and vertical subspaces. Using the connection on the normal bundle of $\Sigma$, we may identify the horizontal subspace at $(\bar{x},\bar{y})$ with the tangent space $T_{\bar{x}} \Sigma$. Moreover, the vertical subspace at $(\bar{x},\bar{y})$ can be identified with $T_{\bar{x}}^\perp \Sigma$. With this understood, the differential $(D\Phi_t)_{(\bar{x},\bar{y})}: T_{\bar{x}} \Sigma \oplus T_{\bar{x}}^\perp \Sigma \to T_{\bar{\gamma}(t)} M$ maps the vector $e_i \in T_{\bar{x}} \Sigma$ to the vector $X_i(t) \in T_{\bar{\gamma}(t)} M$ and the vector $\bar{\gamma}'(0) \in T_{\bar{x}}^\perp \Sigma$ to the vector $t \, \bar{\gamma}'(t) \in T_{\bar{\gamma}(t)} M$. Consequently, $|\det D\Phi_t(\bar{x},\bar{y})| = t \, \det P(t)$ for all $t \in (0,r)$. Putting these facts together, the assertion follows. This completes the proof of Proposition \ref{Riemannian.FWC.monotonicity}. \\

\begin{corollary}
\label{Riemannian.FWC.upper.bound.for.Jacobian}
The Jacobian determinant of $\Phi_r$ satisfies 
\[|\det D\Phi_r(x,y)| \leq r \, \Big ( 1 - r \, \frac{\langle H(x),y \rangle}{n-1} \Big )^{n-1}\] 
for each $r>0$ and each point $(x,y) \in A_r$. 
\end{corollary}

We now continue the proof of Theorem \ref{Riemannian.FWC.estimate}. Using Lemma \ref{Riemannian.FWC.surjectivity} and Corollary \ref{Riemannian.FWC.upper.bound.for.Jacobian}, we obtain 
\begin{align*} 
&\Big | \Big \{ p \in M: \inf_{x \in \Sigma} d(x,p) < r \Big \} \Big | \\ 
&\leq \int_\Sigma \bigg ( \int_{\{y \in T_x^\perp \Sigma: |y| < 1\}} |\det D\Phi_r(x,y)| \, 1_{A_r}(x,y) \, dy \bigg ) \, d\text{\rm vol}(x) \\ 
&\leq \int_\Sigma \bigg ( \int_{\{y \in T_x^\perp \Sigma: |y| < 1\}} r \, \Big ( 1 - r \, \frac{\langle H(x),y \rangle}{n-1} \Big )_+^{n-1} \, dy \bigg ) \, d\text{\rm vol}(x)
\end{align*}
for all $r > 0$. We divide by $r^n$ and send $r \to \infty$. This gives 
\[|B^n| \, \theta \leq \int_\Sigma \bigg ( \int_{\{y \in T_x^\perp \Sigma: |y| < 1\}} \Big ( -\frac{\langle H(x),y \rangle}{n-1} \Big )_+^{n-1} \, dy \bigg ) \, d\text{\rm vol}(x).\] 
On the other hand, since the normal space $T_x^\perp \Sigma$ is one-dimensional, we have 
\[\int_{\{y \in T_x^\perp \Sigma: |y| < 1\}} (-\langle H(x),y \rangle)_+^{n-1} \, dy = \frac{1}{n} \, |H(x)|^{n-1} = \frac{|B^n|}{|S^{n-1}|} \, |H(x)|^{n-1}\] 
for each point $x \in \Sigma$. Putting these facts together, we conclude that 
\[|B^n| \, \theta \leq \frac{|B^n|}{|S^{n-1}|} \int_\Sigma \Big ( \frac{|H(x)|}{n-1} \Big )^{n-1} \, d\text{\rm vol}(x).\] 
This completes the proof of Theorem \ref{Riemannian.FWC.estimate}.


\begin{thebibliography}{99} 
\bibitem{Agostiniani-Fogagnolo-Mazzieri}
V.~Agostiniani, M.~Fogagnolo, and L.~Mazzieri, \textit{Sharp geometric inequalities for closed hypersurfaces in manifolds with nonnegative Ricci curvature,} Invent. Math. 222, 1033--1101 (2020)

\bibitem{Allard}
W.~Allard, \textit{On the first variation of a varifold,} Ann. of Math. 95, 417--491 (1972)

\bibitem{Almgren}
F.J.~Almgren, Jr., \textit{Optimal isoperimetric inequalities,} Indiana Univ. Math. J. 35, 451--547 (1986)

\bibitem{Balogh-Kristaly}
Z.M.~Balogh and A.~Krist\'aly, \textit{Sharp isoperimetric and Sobolev inequalities in spaces with nonnegative Ricci curvature,} Math. Ann. 385, 1747--1773 (2023)

\bibitem{Brendle1}
S.~Brendle, \textit{The isoperimetric inequality for a minimal submanifold in Euclidean space,} J. Amer. Math. Soc. 34, 595--603 (2021)

\bibitem{Brendle2}
S.~Brendle, \textit{The logarithmic Sobolev inequality for a submanifold in Euclidean space,} Comm. Pure Appl. Math. 75, 449--454 (2022)

\bibitem{Brendle3}
S.~Brendle, \textit{Sobolev inequalities in manifolds with nonnegative curvature,} Comm. Pure Appl. Math. 76, 2192--2218 (2023)

\bibitem{Brendle4}
S.~Brendle, \textit{Minimal hypersurfaces and geometric inequalities,} Ann. Fac. Sci. Toulouse 32, 179--201 (2023)

\bibitem{Brendle-Eichmair}
S.~Brendle and M.~Eichmair, \textit{Proof of the Michael-Simon-Sobolev inequality using optimal transport,} J. Reine Angew. Math. 804, 1--10 (2023)

\bibitem{Cabre1}
X.~Cabr\'e, \textit{Nondivergent elliptic equations on manifolds with nonnegative curvature,} Comm. Pure Appl. Math. 50, 623--665 (1997)

\bibitem{Cabre2}
X.~Cabr\'e, \textit{Equacions en derivades parcials, geometria i control estoc\`astic,} Butl. Soc. Catalana Mat. 15, 7--27 (2000)

\bibitem{Cabre3}
X.~Cabr\'e, \textit{Elliptic PDEs in probability and geometry. Symmetry and regularity of solutions,} Discrete Cont. Dyn. Systems 20, 425--457 (2008)

\bibitem{Carleman} 
T.~Carleman, \textit{Zur Theorie der Minimalfl\"achen,} Math. Z. 9, 154--160 (1921)

\bibitem{Castillon}
P.~Castillon, \textit{Submanifolds, isoperimetric inequalities and optimal transportation,} J. Funct. Anal. 259, 79--103 (2010)

\bibitem{Chen}
B.-Y.~Chen, \textit{On the total curvature of immersed manifolds, I: an inequality of Fenchel-Borsuk-Willmore,} Amer. J. Math. 93, 148--162 (1971)

\bibitem{Choe}
J.~Choe, \textit{The isoperimetric inequality for a minimal surface with radially connected boundary,} Ann. Scuola Norm. Sup. Pisa 17, 583--593 (1990)

\bibitem{Cordero-Erausquin-McCann-Schmuckenschlager}
D.~Cordero-Erausquin, R.J.~McCann, and M.~Schmuckenschl\"ager, \textit{A Riemannian interpolation inequality \`a la Borell, Brascamp and Lieb,} Invent. Math. 146, 219--257 (2001)

\bibitem{Ecker}
K.~Ecker, \textit{Logarithmic Sobolev inequalities on submanifolds of Euclidean space,} J. Reine Angew. Math. 522, 105--118 (2000)

\bibitem{Feinberg}
J.~Feinberg, \textit{The isoperimetric inequality for doubly-connected minimal surfaces in $\mathbb{R}^n$,} J. d'Anal. Math. 32, 249--278 (1977)

\bibitem{Fenchel}
W.~Fenchel, \textit{\"Uber die Kr\"ummung und Windung geschlossener Randkurven,} Math. Ann. 101, 238--252 (1929)

\bibitem{Gromov}
M.~Gromov, \textit{Metric structures for Riemannian and non-Riemannian spaces,} Progress in Mathematics vol. 152, Birkh\"auser, Boston, 1999

\bibitem{Heintze-Karcher}
E.~Heintze and H.~Karcher, \textit{A general comparison theorem with applications to volume estimates for submanifolds,} Ann. Sci. \'Ecole Norm. Sup. 11, 451--470 (1978)

\bibitem{Hsiung}
C.C.~Hsiung, \textit{Isoperimetric inequalities for two-dimensional Riemannian manifolds with boundary,} Ann. of Math. 73, 213--220 (1961)

\bibitem{Huisken}
G.~Huisken, \textit{An isoperimetric concept for the mass in general relativity,} Lecture given at the Institute for Advanced Study on March 20, 2009 
%\begin{verbatim}
%https://www.ias.edu/video/marston-morse-isoperimetric-concept-mass-general-relativity
%\end{verbatim}

\bibitem{Li-Schoen-Yau}
P.~Li, R.~Schoen, and S.T.~Yau, \textit{On the isoperimetric inequality for minimal surfaces,} Ann. Scuola Norm. Sup. Pisa 11, 237--244 (1984)

\bibitem{Lusternik}
L.A.~Lusternik, \textit{Die Brunn-Minkowskische Ungleichung f\"ur beliebige me\ss{}bare Mengen,} Dokl. Acad. Sci. URSS 8, 55--58 (1935)

\bibitem{McCann1}
R.J.~McCann, \textit{A Convexity Theory for Interacting Gases and Equilibrium Crystals,} Ph.D. thesis, Princeton University, 1994

\bibitem{McCann2}
R.J.~McCann, \textit{A convexity principle for interacting gases,} Adv. Math. 128, 153–179 (1997)

\bibitem{McCann-Guillen}
R.J.~McCann and N.~Guillen, \textit{Five lectures on optimal transportation: geometry, regularity, and applications,} Analysis and Geometry of Metric Measure Spaces, CRM Proc. Lecture Notes vol. 56, pp.~145--180, American Mathematical Society, Providence RI, 2013

\bibitem{Michael-Simon}
J.H.~Michael and L.M.~Simon, \textit{Sobolev and mean value inequalities on generalized submanifolds of $\mathbb{R}^n$,} Comm. Pure Appl. Math. 26, 361--379 (1973)

\bibitem{Osserman-Schiffer}
R.~Osserman and M.~Schiffer, \textit{Doubly-connected minimal surfaces,} Arch. Rational Mech. Anal. 58, 285--307 (1975)

\bibitem{Reid}
W.T.~Reid, \textit{The isoperimetric inequality and associated boundary problems,} J. Math. Mech. 8, 897--906 (1959)

\bibitem{Stein-Shakarchi1}
E.M.~Stein and R.~Shakarchi, \textit{Complex Analysis,} Princeton University Press, 2003

\bibitem{Stein-Shakarchi2}
E.M.~Stein and R.~Shakarchi, \textit{Real analysis: measure theory, integration, and Hilbert spaces,} Princeton University Press, 2009

\bibitem{Stone}
A.~Stone, \textit{On the isoperimetric inequality on a minimal surface,} Calc. Var. PDE 17, 369--391 (2003)

\bibitem{Trudinger}
N.~Trudinger, \textit{Isoperimetric inequalities for quermassintegrals,} Ann. Inst. H. Poincar\'e Anal. Non Lin\'eaire 11, 411--425 (1994)

\bibitem{Willmore}
T.J.~Willmore, \textit{Note on embedded surfaces,} A. Sti. Univ. "Al. I. Cuza", Iasi, Sect. Ia Mat. 11B, 493--496 (1965)
\end{thebibliography}
\end{document}